\documentclass[12pt]{article}
\usepackage{amssymb}
\usepackage{amsthm}
\usepackage{amsmath}
\usepackage{graphicx}
\usepackage{enumerate}

\DeclareMathOperator{\Pro}{Pr}
\DeclareMathOperator{\BPr}{\mathbf{Pr}}
\newtheorem{theorem}{Theorem}[section]
\newtheorem{definition}[theorem]{Definition}
\newtheorem{lemma}[theorem]{Lemma}
\newtheorem{proposition}[theorem]{Proposition}

\newtheorem{example}[theorem]{Example}
\newtheorem{corollary}[theorem]{Corollary}
\title{Semimodules over commutative semirings and modules over unitary commutative rings}
\author{Ivan~Chajda and Helmut~L\"anger}
\date{}
\begin{document}
\footnotetext[1]{Support of the research by \"OAD, project CZ~02/2019, and support of the research of the first author by IGA, project P\v rF~2019~015, is gratefully acknowledged.}
\maketitle
\begin{abstract}
We study the so-called closed and splitting subsemimodules and submodules of a given semimodule or module, respectively. We describe lattices of subsemimodules and of closed subsemimodules and posets of splitting subsemimodules and submodules. In the case of modules a natural bijective correspondence between these posets and posets of projections is established.
\end{abstract}
 


\section{Introduction}

It is well-known that any physical theory determines a class of event-state systems. To avoid details, in the case of quantum mechanics this event-state system is considered within the framework of a Hilbert space $\mathbf H$ whose projection operators are identified with the closed subspaces of $\mathbf H$.

It was recognized in 1936 by G.~Birkhoff and J.~von~Neumann (\cite{BV}) and 1937 by K.~Husimi (\cite H), see also \cite M or \cite Z, that if the Hilbert space $\mathbf H$ is of infinite dimension then the lattice of its closed subspaces need not be modular contrary to the case of the lattice of all subspaces. However, a later inspection showed that also a supremum need not exist provided the subspaces are orthogonal. This was the reason why so-called orthomodular posets were introduced (see e.g.\ \cite{Be}) and intensively studied during the last decades.

The natural question arises if the property that the closed subspaces of $\mathbf H$ form an orthomodular lattice or an orthomodular poset is a privilege of a Hilbert space. It was already shown by the authors \cite{CL20} that this is not the case since the so-called splitting subspaces form orthomodular posets also for vector spaces which are not Hilbert spaces.

Since the tools for determining the orthomodular poset of splitting subspaces of a given vector space can be used also for modules and, more generally, for semimodules as shown in \cite E and \cite{Tan}, we decided to extend our study for closed subsemimodules and submodules. We define splitting subsemimodules and prove that for a given semimodule $\mathbf M$, the set of its splitting subsemimodules forms a bounded poset with an antitone involution which, in the case when $\mathbf M$ is a module, turns out to be even an orthomodular poset. Similarly as for a Hilbert space, we use the method of projections and the bijective correspondence between the poset of projections and the poset of splitting submodules.

The used concepts from posets (i.e.\ ordered sets) and lattices are taken from monographs \cite{Be} and \cite{Bi}. We hope that the study of closed and splitting subsemimodules and submodules and their lattices and posets can illuminate some properties of these concepts also in vector spaces, in particular in Hilbert spaces. Moreover, it may show that some physical theories need not be developed by using Hilbert spaces, but can be considered in a more general setting.

\section{Semimodules over semirings}

There are various definitions of a semiring in literature. For our reasons, we use that taken from the monograph \cite G.

Recall that a {\em commutative semiring} is an algebra $(S,\oplus,\cdot,0,1)$ of type $(2,2,0,$ $0)$ satisfying the following conditions:
\begin{itemize}
\item $(S,\oplus,0)$ and $(S,\cdot,1)$ are commutative monoids,
\item $(x\oplus y)z\approx xz\oplus yz$,
\item $x0\approx0$.
\end{itemize}

Of course, every unitary commutative ring and every bounded distributive lattice is a commutative semiring.

Semimodules and semirings were studied by several authors, let us mention at least the papers \cite E, \cite{SW}, \cite{Tak} and \cite{Tan}. Since these concepts are defined differently by the different authors, for the reader's convenience we provide the following definition.

\begin{definition}
A {\em semimodule} over a commutative semiring $(S,\oplus,\cdot,0,1)$ is an ordered quadruple $(M,+,\cdot,\vec0)$ such that $\cdot$ is a mapping from $S\times M$ to $M$ and the following conditions are satisfied for $\vec x,\vec y\in M$ and $a,b\in S$:
\begin{itemize}
\item $(M,+,\vec0)$ is a commutative monoid,
\item $a(\vec x+\vec y)\approx a\vec x+a\vec y$,
\item $(a\oplus b)\vec x\approx a\vec x+b\vec x$,
\item $(ab)\vec x\approx a(b\vec x)$,
\item $1\vec x\approx\vec x$,
\item $0\vec x\approx a\vec0=\vec0$.
\end{itemize}
\end{definition}

Recall that a subset $U$ of a semimodule $\mathbf M=(M,+,\cdot,\vec0)$ over a commutative semiring $(S,\oplus,\cdot,0,1)$ {\rm(}or the corresponding ordered quadruple $(U,+,\cdot,\vec0)${\rm)} is called a {\em subsemimodule} of $\mathbf M$ if $\vec x+\vec y,a\vec x\in U$ for all $\vec x,\vec y\in U$ and $a\in S$. Let $L(\mathbf M)$ denote the set of all subsemimodules of $\mathbf M$.

Contrary to the case of vector spaces, not every semimodule may have a basis. We define the notion of a basis for semimodules as follows.

\begin{definition}
Let $\mathbf M=(M,+,\cdot,\vec0)$ be a semimodule over a commutative semiring $(S,\oplus,\cdot,0,1)$ and $I$ a non-empty set, put
\[
A:=\{f\in S^I\mid f(i)=0\text{ except a finite number of }i\in I\}
\]
and let $\vec b_i\in M$ for all $i\in I$. Then $B:=\{\vec b_i\mid i\in I\}$ is called a {\em basis} of $\mathbf M$ if for every $\vec x\in M$ there exists exactly one $f\in A$ with
\[
\sum_{i\in I}f(i)\vec b_i=\vec x.
\]
\end{definition}

In the following we will assume that $\mathbf M$ has a basis $B$. Then $\mathbf M$ is isomorphic to the subsemimodule $(A,+,\cdot,\vec0)$ of $(S,\oplus,\cdot,0)^I$. Hence we may identify $\mathbf M$ with this subsemimodule. In the sequel we denote the coordinates of the element $\vec x$ of $M$ with respect to the basis $B=\{\vec b_i\mid i\in I\}$ by $x_i,i\in I$.

An example of a semimodule having a basis is the following.

If, for instance, $(S,\oplus,\cdot,0,1)$ is an arbitrary commutative semiring and $I=\mathbb N$ then the subsemimodule $(A,+,\cdot,\vec0)$ of $(S,\oplus,\cdot,0)^\mathbb N$ has the basis
\[
\{(1,0,0,\ldots),(0,1,0,0,\ldots),(0,0,1,0,0,\ldots),\ldots\}.
\]
The situation is analogous for an arbitrary non-empty set $I$.

The concept of an inner product on semimodules was investigated in \cite{Tan}. For the reader's convenience we recall the definition of the inner product as well as the concept of orthogonality for subsemimodules.

\begin{definition}
On $M$ we define an {\em inner product} as follows: If $\vec x,\vec y\in M$ then
\[
\vec x\vec y:=\sum\limits_{i\in I}x_iy_i.
\]
We write $\vec x\perp\vec y$ id $\vec x\vec y=0$. Moreover, for $C\subseteq M$ we put
\[
C^\perp:=\{\vec x\in M\mid \vec x\perp\vec y\text{ for all }\vec y\in C\}.
\]
\end{definition}

\begin{lemma}\label{lem2}
Let $\vec a,\vec b\in M$. Then {\rm(i)} and {\rm(ii)} hold:
\begin{enumerate}
\item[{\rm(i)}] If $\vec a\vec x=\vec b\vec x$ for all $\vec x\in M$ then $\vec a=\vec b$,
\item[{\rm(ii)}] if $\vec a\perp\vec x$ for all $\vec x\in M$ then $\vec a=\vec0$,
\end{enumerate}
\end{lemma}

\begin{proof}
We have $a_i=\vec a\vec b_i=\vec b\vec b_i=b_i$ for all $i\in I$. Assertion (ii) is a special case of (i).
\end{proof}

The following results are well-known and easy to check.

\begin{proposition}\label{prop1}
If $U,W\in L(\mathbf M)$ then
\begin{itemize}
\item $U^\perp\in L(\mathbf M)$,
\item $U\subseteq W$ implies $W^\perp\subseteq U^\perp$,
\item $U\subseteq U^{\perp\perp}$,
\item $U^{\perp\perp\perp}=U^\perp$,
\item $U\subseteq W^\perp$ if and only if $W\subseteq U^\perp$,
\item $\{\vec0\}^\perp=M$ and $M^\perp=\{\vec0\}$.
\end{itemize}
{\rm(}The last assertion follows from Lemma~\ref{lem2}.{\rm)} Thus $^{\perp\perp}$ is a closure operator on $(L(\mathbf M),\subseteq)$.
\end{proposition}

\begin{definition}\label{def1}
A subsemimodule $U$ of $\mathbf M$ is called {\em closed} if $U^{\perp\perp}=U$. Let $L_c(\mathbf M)$ denote the set of all closed subsemimodules of $\mathbf M$. Obviously, $L_c(\mathbf M)=\{U^\perp\mid U\in L(\mathbf M)\}$.
\end{definition}

Let $U,W,U_j\in L(\mathbf M)$ for all $j\in J$. Put
\begin{align*}
                   U+W & :=\{\vec x+\vec y\mid\vec x\in U,\vec y\in W\}, \\
      \sum_{j\in J}U_j & :=\{\text{sums of finitely many elements of }\bigcup_{j\in J}U_j\}, \\
               U\vee W & :=(U+W)^{\perp\perp}, \\
   \bigvee_{j\in J}U_j & :=(\sum_{j\in J}U_j)^{\perp\perp}, \\
  \mathbf L(\mathbf M) & :=(L(\mathbf M),+,\cap,{}^\perp,\{\vec0\},M), \\
\mathbf L_c(\mathbf M) & :=(L_c(\mathbf M),\vee,\cap,{}^\perp,\{\vec0\},M).
\end{align*}

We can describe the properties of the just defined concepts as follows.

\begin{lemma}\label{lem3}
\
\begin{enumerate}
\item[{\rm(i)}] If $U_j\in L(\mathbf M)$ for all $j\in J$ then
\begin{align*}
   (\sum\limits_{j\in J}U_j)^\perp & =\bigcap\limits_{j\in J}U_j^\perp, \\
(\bigcap\limits_{j\in J}U_j)^\perp & \supseteq\sum\limits_{j\in J}U_j^\perp.
\end{align*}
\item[{\rm(ii)}] If $U_j\in L_c(\mathbf M)$ for all $j\in J$ then
\begin{align*}
(\bigvee_{j\in J}U_j)^\perp & =\bigcap_{j\in J}U_j^\perp, \\
(\bigcap_{j\in J}U_j)^\perp & =\bigvee_{j\in J}U_j^\perp.
\end{align*}
\end{enumerate}
\end{lemma}

\begin{proof}
\
\begin{enumerate}
\item[(i)] The first assertion is clear and the second easily follows by applying Proposition~\ref{prop1}.
\item[(ii)] This follows from the fact that by Proposition~\ref{prop1}, $^\perp$ is an antitone involution of $(L_c(\mathbf M),\subseteq)$.
\end{enumerate}
\end{proof}

Using Lemma~\ref{lem3} we obtain immediately

\begin{theorem}\label{th1}
We have that $\mathbf L(\mathbf M)$ is a complete lattice with an antitone unary operation $^\perp$ and $\mathbf L_c(\mathbf M)$ a complete lattice with an antitone involution $^\perp$.
\end{theorem}

\begin{proof}
This follows from Proposition~\ref{prop1} and Lemma~\ref{lem3}.
\end{proof}

The lattices $\mathbf L(\mathbf M)$ and $\mathbf L_c(\mathbf M)$ are related as shown in the next theorem.

\begin{theorem}\label{th2}
\
\begin{enumerate}
\item[{\rm(i)}] Assume $(U\cap W)^{\perp\perp}=U^{\perp\perp}\cap W^{\perp\perp}$ for all $U,W\in L(\mathbf M)$. Then $^{\perp\perp}$ is a surjective homomorphism from $\mathbf L(\mathbf M)$ to $\mathbf L_c(\mathbf M)$.
\item[{\rm(ii)}] Assume
\[
(\bigcap\limits_{j\in J}U_j)^{\perp\perp}=\bigcap\limits_{j\in J}U_j^{\perp\perp}
\]
for every family $(U_j;j\in J)$ of subsemimodules of $\mathbf M$. Then $^{\perp\perp}$ is a complete surjective homomorphism from $\mathbf L(\mathbf M)$ to $\mathbf L_c(\mathbf M)$.
\end{enumerate}
\end{theorem}

\begin{proof}
Let $U,W,U_j\in L(\mathbf M)$ for all $j\in J$.
\begin{enumerate}
\item[(i)] We have
\begin{align*}
    (U+W)^{\perp\perp} & =(U^\perp\cap W^\perp)^\perp=(U^{\perp\perp\perp}\cap W^{\perp\perp\perp})^\perp=(U^{\perp\perp}+W^{\perp\perp})^{\perp\perp}=U^{\perp\perp}\vee W^{\perp\perp}, \\
(U\cap W)^{\perp\perp} & =U^{\perp\perp}\cap W^{\perp\perp}, \\
(U^\perp)^{\perp\perp} & =(U^{\perp\perp})^\perp, \\
\{\vec0\}^{\perp\perp} & =\{\vec0\}, \\
        M^{\perp\perp} & =M.
\end{align*}
\item[(ii)] We have
\begin{align*}
   (\sum_{j\in J}U_j)^{\perp\perp} & =(\bigcap_{j\in J}U_j^\perp)^\perp=(\bigcap_{j\in J}U_j^{\perp\perp\perp})^\perp=(\sum_{j\in J}U_j^{\perp\perp})^{\perp\perp}=\bigvee_{j\in J}U_j^{\perp\perp}, \\
(\bigcap_{j\in J}U_j)^{\perp\perp} & =\bigcap_{j\in J}U_j^{\perp\perp}, \\
            (U^\perp)^{\perp\perp} & =(U^{\perp\perp})^\perp, \\
            \{\vec0\}^{\perp\perp} & =\{\vec0\}, \\
                    M^{\perp\perp} & =M.
\end{align*}
\end{enumerate}
\end{proof}

\begin{example}\label{ex1}
Consider the semiring $(S,\oplus,\cdot,0,1)$ where $S=\{0,1\}$ and the operations $\oplus$ and $\cdot$ are determined by the tables
\[
\begin{array}{c|cc}
\oplus & 0 & 1 \\
\hline
   0   & 0 & 1 \\
   1   & 1 & 1
\end{array}
\quad\quad\quad
\begin{array}{c|ccc}
\cdot & 0 & 1 \\
\hline
  0   & 0 & 0 \\
  1   & 0 & 1
\end{array}
\]
Put $\mathbf M:=(S,\oplus,\cdot,0)^2$. Then $\mathbf M$ has the following subspaces:
\begin{align*}
U_1 & =\{(0,0)\}, \\
U_2 & =\{(0,0),(0,1)\}, \\
U_3 & =\{(0,0),(1,1)\}, \\
U_4 & =\{(0,0),(1,0)\}, \\
U_5 & =\{(0,0),(0,1),(1,1)\}, \\
U_6 & =\{(0,0),(1,0),(1,1)\}, \\
  M & =\{(0,0),(0,1),(1,0),(1,1)\}.
\end{align*}
The Hasse diagram of $(L(\mathbf M),\subseteq)$ is presented in Figure~1:

\vspace*{-2mm}

\[
\setlength{\unitlength}{7mm}
\begin{picture}(10,9)
\put(5,2){\circle*{.3}}
\put(1,4){\circle*{.3}}
\put(5,4){\circle*{.3}}
\put(9,4){\circle*{.3}}
\put(3,6){\circle*{.3}}
\put(7,6){\circle*{.3}}
\put(5,8){\circle*{.3}}
\put(5,2){\line(-2,1)4}
\put(5,2){\line(2,1)4}
\put(5,2){\line(0,1)2}
\put(5,8){\line(-1,-1)4}
\put(5,8){\line(1,-1)4}
\put(5,4){\line(1,1)2}
\put(5,4){\line(-1,1)2}
\put(4.675,1.25){$U_1$}
\put(.05,3.85){$U_2$}
\put(5.4,3.85){$U_3$}
\put(9.4,3.85){$U_4$}
\put(2.05,5.85){$U_5$}
\put(7.4,5.85){$U_6$}
\put(4.7,8.35){$M$}
\put(4.2,.3){{\rm Fig.~1}}
\end{picture}
\]

\vspace*{-3mm}

The lattice $\mathbf L(\mathbf M)$ is not modular because it contains sublattices isomorphic to ${\rm N}_5$, e.g.\ the sublattice $\{U_1,U_2,U_4,U_6,M\}$. The unary operation $^\perp$ looks as follows:
\[
\begin{array}{c|ccccccccc}
   U    & U_1 & U_2 & U_3 & U_4 & U_5 & U_6 &  M \\
\hline
U^\perp &  M  & U_4 & U_1 & U_2 & U_1 & U_1 & U_1
\end{array}
\]
Hence, $L_c(\mathbf M)=\{U_1,U_2,U_4,M\}$. The Hasse diagram of $(L_c(\mathbf M),\subseteq)$ is depicted in Figure~2:

\vspace*{-6mm}

\[
\setlength{\unitlength}{7mm}
\begin{picture}(6,7)
\put(3,2){\circle*{.3}}
\put(1,4){\circle*{.3}}
\put(5,4){\circle*{.3}}
\put(3,6){\circle*{.3}}
\put(3,2){\line(-1,1)2}
\put(3,2){\line(1,1)2}
\put(3,6){\line(-1,-1)2}
\put(3,6){\line(1,-1)2}
\put(2.675,1.25){$U_1$}
\put(.05,3.85){$U_2$}
\put(5.4,3.85){$U_4$}
\put(2.7,6.35){$M$}
\put(2.2,.3){{\rm Fig.~2}}
\end{picture}
\]

\vspace*{-3mm}

\end{example}

\section{Splitting subsemimodules}

It can be easily checked that for a subsemimodule $U$ of $\mathbf M$, the semimodule $U^\perp$ need not be a complement of $U$ in the lattice $\mathbf L(\mathbf M)$ or $\mathbf L_c(\mathbf M)$, see e.g.\ Example~\ref{ex1}. This is the motivation for introducing the following concept.

\begin{definition}
We call a {\em subsemimodule} $U$ of $\mathbf M$ {\em splitting} if $U+U^\perp=M$ and $U\cap U^\perp=\{\vec0\}$. Let $L_s(\mathbf M)$ denote the set of all splitting subspaces of $\mathbf M$.
\end{definition}

Clearly, $\{\vec0\},M\in L_s(\mathbf M)$.

\begin{example}
The splitting subsemimodules of the semimodule from Example~\ref{ex1} are exactly the closed ones.
\end{example}

\begin{lemma}
Every splitting subsemimodule of $\mathbf M$ is closed.
\end{lemma}

\begin{proof}
Assume $U\in L_s(\mathbf M)$, $\vec a\in U^{\perp\perp}$ and $\vec d\in M$. Then there exist $\vec b,\vec e\in U$ and $\vec c,\vec f\in U^\perp$ with $\vec b+\vec c=\vec a$ and $\vec e+\vec f=\vec d$. Since
\[
\vec a\in U^{\perp\perp},\vec b,\vec e\in U\text{ and }\vec c,\vec f\in U^\perp,
\]
we have
\[
\vec a\vec f=\vec c\vec e=\vec b\vec f=0
\]
and hence
\[
\vec a\vec d=\vec a(\vec e+\vec f)=\vec a\vec e+\vec a\vec f=\vec a\vec e=(\vec b+\vec c)\vec e=\vec b\vec e+\vec c\vec e=\vec b\vec e=\vec b\vec e+\vec b\vec f=\vec b(\vec e+\vec f)=\vec b\vec d.
\]
According to Lemma~\ref{lem2}, $\vec a=\vec b\in U$. This shows $U^{\perp\perp}\subseteq U$. The converse inclusion follows from Proposition~\ref{prop1}.
\end{proof}

Recall that if $(P,\leq,0,1)$ is a bounded poset, then a unary operation $'$ on $P$ is called a {\em complementation} if $\sup(x,x')=1$ and $\inf(x,x')=0$ for all $x\in P$. If $'$ is, moreover, an antitone involution then $(P,\leq,{}',0,1)$ is called an {\em orthoposet}. In the sequel, we will denote $\sup$ and $\inf$ by $\vee$ and $\wedge$, respectively, provided they exist.

\begin{corollary}
We have that $\mathbf L_s(\mathbf M):=(L_s(\mathbf M),\subseteq,{}^\perp,\{\vec0\},M)$ is an orthoposet.
\end{corollary}

It is a question if the poset $(L_s(\mathbf M),\subseteq)$ of splitting subsemimodules of $\mathbf M$ is a lattice depending of the choice of the semiring $\mathbf S$. It turns out that in some particular cases this is true.

Assume that $\mathbf S=(S,\vee,\wedge,0,1)$ is a non-trivial bounded distributive lattice where $0$ is meet-irreducible, i.e.\ $x\wedge y=0$ implies $0\in\{x,y\}$, let $I$ be a non-empty set, put
\[
M:=\{\vec x\in S^I\mid x_i=0\text{ for almost all }i\in I\}
\]
and consider the submodule $\mathbf M=(M,\vee,\wedge,\vec0)$ of $(S,\vee,\wedge,0)^I$. For every subset $J$ of $I$ put $U_J:=\{\vec x\in M\mid x_i=0\text{ for all }i\in J\}$.

A mapping $f$ from a poset $(P,\leq)$ to a poset $(Q,\leq)$ is called an {\em antiisomorphism} if it is bijective and if for all $x,y\in P$, $x\leq y$ is equivalent to $f(y)\leq f(x)$.

Now we can prove the following.

\begin{theorem}
Let $(S,\vee,\wedge,0,1)$ be a non-trivial bounded distributive lattice where $0$ is meet-irreducible and put
\[
M:=\{\vec x\in S^I\mid x_i=0\text{ for almost all }i\in I\}
\]
for a non-empty set $I$. Then $(L_s(\mathbf M),\subseteq)=(L_c(\mathbf M),\subseteq)$ is an atomic Boolean algebra and the mapping $J\mapsto U_J$ an antiisomorphism between the posets $(2^I,\subseteq)$ and $(L_s(\mathbf M),\subseteq)$.
\end{theorem}

\begin{proof}
It is clear that for $\vec a,\vec b\in M$ we have $\vec a\perp\vec b$ if and only if for all $i\in I$ either $a_i=0$ or $b_i=0$ (or both). Hence, for $U\in L(\mathbf M)$ we have $U^\perp=U_K$ where
\[
K=\{i\in I\mid\text{there exists some }\vec x\in U\text{ with }x_i\neq0\}.
\]
Obviously, $U_J^\perp=U_{I\setminus J}$ for all $J\subseteq I$. This shows $L_c(\mathbf M)=\{U_J\mid J\subseteq I\}$. Now let $S,T\subseteq I$. If $S\subseteq T$ then $U_T\subseteq U_S$. Conversely, assume $U_T\subseteq U_S$. Suppose $S\not\subseteq T$. Then there exists some $j\in S\setminus T$. Let $\vec a$ denote the element of $M$ with $a_j=1$ and $a_i=0$ otherwise. Then $\vec a\in U_T\setminus U_S$ contradicting $U_T\subseteq U_S$. Hence $S\subseteq T$. This shows that $S\subseteq T$ is equivalent to $U_T\subseteq U_S$ completing the proof of the theorem.
\end{proof}

It should be remarked that in any non-trivial bounded chain the smallest element is meet-irreducible.

\section{The poset of projections}

The next concept plays a crucial role in our study.

\begin{definition}
A {\em projection} of $\mathbf M$ is a linear mapping $P$ from $\mathbf M$ to $\mathbf M$ satisfying $P\circ P=P$ and $(P\vec x)\vec y=\vec x(P\vec y)$ for all $\vec x,\vec y\in M$. We write $P\vec x$ instead of $P(\vec x)$. Let $\Pro(\mathbf M)$ denote the set of all projections of $\mathbf M$ and $P,Q\in\Pro(\mathbf M)$. We define $P\leq Q$ if $P(M)\subseteq Q(M)$, and, moreover, $(P+Q)(\vec x):=P\vec x+Q\vec x$ and $PQ\vec x:=P(Q(\vec x))$ for all $\vec x\in M$. Let $\mathbf0$ denote the constant mapping from $M$ to $M$ with value $\vec0$ and $\mathbf I$ the identical mapping from $M$ to $M$.
\end{definition}

Clearly, $\mathbf0,\mathbf I\in\Pro(\mathbf M)$.

\begin{lemma}\label{lem5}
Let $P,Q\in\Pro(\mathbf M)$.
\begin{enumerate}
\item[{\rm(i)}] The following are equivalent:
\begin{enumerate}
\item[{\rm(a)}] $P\leq Q$,
\item[{\rm(b)}] $PQ=P$,
\item[{\rm(c)}] $QP=P$.
\end{enumerate}
\item[{\rm(ii)}] Assume $PQ=QP$. Then the infimum $P\wedge Q$ exists and $P\wedge Q=PQ$.
\end{enumerate}
\end{lemma}

\begin{proof}
Let $\vec a,\vec b\in M$.
\begin{enumerate}
\item[(i)] (a) $\Rightarrow$ (b): Since $P\vec b\in P(M)\subseteq Q(M)$, there exists some $\vec c\in M$ with $P\vec b=Q\vec c$. Now
\[
(PQ\vec a)\vec b=(Q\vec a)(P\vec b)=(Q\vec a)(Q\vec c)=\vec a(Q^2\vec c)=\vec a(Q\vec c)=\vec a(P\vec b)=(P\vec a)\vec b
\]
showing $PQ=P$. \\
(b) $\Rightarrow$ (c): We have
\[
(QP\vec a)\vec b=(P\vec a)(Q\vec b)=\vec a(PQ\vec b)=\vec a(P\vec b)=(P\vec a)\vec b
\]
showing $QP=P$. \\
(c) $\Rightarrow$ (a): We have $P(M)=QP(M)\subseteq Q(M)$.
\item[(ii)] Let $R\in\Pro(\mathbf M)$. Obviously, $PQ$ is a linear mapping from $\mathbf M$ to itself. Moreover,
\begin{align*}
          (PQ)^2 & =PQPQ=P^2Q^2=PQ, \\
(PQ\vec a)\vec b & =(Q\vec a)(P\vec b)=\vec a(QP\vec b)=\vec a(PQ\vec b)
\end{align*}
showing $PQ\in\Pro(\mathbf M)$. Now $(PQ)P=PQ$, i.e.\ $PQ\leq P$, and $(PQ)Q=PQ$, i.e.\ $PQ\leq Q$. Moreover, if $R\leq P,Q$ then $R(PQ)=(RP)Q=RQ=R$ and hence $R\leq PQ$. This shows $PQ=P\wedge Q$.
\end{enumerate}
\end{proof}

Moreover, we can prove the following.

\begin{theorem}\label{th3}
Let $\mathbf M$ be a semimodule. Then $(\Pro(\mathbf M),\leq,\mathbf0,\mathbf I)$ is a bounded poset.
\end{theorem}

\begin{proof}
We apply Lemma~\ref{lem5}. Let $P,Q,R\in\Pro(\mathbf M)$. Since $P^2=P$ we have $P\leq P$, if $P\leq Q\leq P$ then $P=PQ=Q$, and if $P\leq Q\leq R$ then
\[
PR=(PQ)R=P(QR)=PQ=P,
\]
i.e.\ $P\leq R$. Thus, $(\Pro(\mathbf M),\leq)$ is a poset. Clearly, $\mathbf0\leq P\leq\mathbf I$.
\end{proof}

It is elementary to check the following Proposition.

\begin{proposition}
The mapping $P\mapsto P(M)$ is a order homomorphism from the bounded poset $(\Pro(\mathbf M),\leq,\mathbf0,\mathbf I)$ to the bounded poset $(L(\mathbf M),\subseteq,\{\vec0\},M)$.
\end{proposition}

\section{Modules over rings}

In this section we will investigate modules over unitary commutative rings instead of semimodules over commutative semirings. Of course, every module $\mathbf M$ over a unitary commutative ring $\mathbf S$ is a semimodule but now $(M,+)$ is a commutative group. It means that on $M$ there is also a binary operation $-$ of subtraction. This enables us to reach stronger results than those above for semimodules.

In the sequel we assume that the semimodule $\mathbf M$ over the commutative semiring $\mathbf S$ is a module over the unitary commutative ring $\mathbf S$, i.e.\ $(M,+)$ is a commutative group.

In this section let $L(\mathbf M)$, $L_c(\mathbf M)$ and $L_s(\mathbf M)$ denote the set of all submodules, closed submodules and splitting submodules of $\mathbf M$, respectively.

The following result is well-known:

For every module $\mathbf M$, the lattice $\mathbf L(\mathbf M)$ is modular contrary to the case of semimodules, see Example~\ref{ex1}.

\begin{definition}
Let $P,Q\in\Pro(\mathbf M)$. We define $(P-Q)\vec x:=P\vec x-Q\vec x$ for all $\vec x\in M$. Further, $P':=\mathbf I-P$ and $P\perp Q$ if $P\leq Q'$.
\end{definition}

\begin{lemma}\label{lem4}
Let $P,Q\in\Pro(\mathbf M)$. Then $P'\in\Pro(\mathbf M)$, and $P\perp Q\Leftrightarrow PQ=\mathbf0\Leftrightarrow QP=\mathbf0$.
\end{lemma}

\begin{proof}
Let $\vec a,\vec b\in M$. Clearly, $P'$ is a linear mapping from $\mathbf M$ to itself,
\begin{align*}
          (P')^2 & =(\mathbf I-P)(\mathbf I-P)=\mathbf I-P-P+P^2=\mathbf I-P=P', \\
(P'\vec a)\vec b & =((\mathbf I-P)\vec a)\vec b=(\vec a-P\vec a)\vec b=\vec a\vec b-(P\vec a)\vec b=\vec a\vec b-\vec a(P\vec b)=\vec a(\vec b-P\vec b)= \\
                 & =\vec a((\mathbf I-P)\vec b)=\vec a(P'\vec b)
\end{align*}
showing $P'\in\Pro(\mathbf M)$. Finally,
\[
P\perp Q\Leftrightarrow P\leq Q'\Leftrightarrow PQ'=P\Leftrightarrow Q'P=P.
\]
\end{proof}

By Theorem~\ref{th3}, $(\Pro(\mathbf M),\leq,\mathbf0,\mathbf I)$ is a bounded poset. Now we can prove for modules a bit more.

\begin{lemma}\label{lem1}
Let $\mathbf M$ be a module. Then $\BPr(\mathbf M):=(\Pro(\mathbf M),\leq,{}',\mathbf0,\mathbf I)$ is a bounded poset with an antitone involution.
\end{lemma}

\begin{proof}
Let $P,Q\in\Pro(\mathbf M)$. If $P\leq Q$ then
\[
Q'P'=(\mathbf I-Q)(\mathbf I-P)=\mathbf I-P-Q+QP=\mathbf I-Q=Q'
\]
according to Lemma~\ref{lem5}, i.e.\ $Q'\leq P'$. Finally, $P''=\mathbf I-(\mathbf I-P)=P$.
\end{proof}

For a splitting submodule $U$ of a module $\mathbf M$ we can show now that every element of $M$ can be uniquely decomposed into a sum of two elements, one belonging to $U$ and the other to $U^\perp$.

\begin{lemma}\label{lem7}
Let $U\in L_s(\mathbf M)$ and $\vec a\in M$. Then there exist unique $b\in U$ and $c\in U^\perp$ with $\vec b+\vec c=\vec a$.
\end{lemma}

\begin{proof}
Because of $M=U+U^\perp$ there exist $\vec b\in U$ and $\vec c\in U^\perp$ with $\vec b+\vec c=\vec a$. If $\vec d\in U$, $\vec e\in U^\perp$ and $\vec d+\vec e=\vec a$ then $\vec b-\vec d=\vec e-\vec c\in U\cap U^\perp=\{\vec0\}$ and hence $(\vec b,\vec c)=(\vec d,\vec e)$.
\end{proof}

If $U\in L_s(\mathbf M)$ then let $P_U$ denote the unique mapping from $M$ to $M$ with $P_U(\vec x)\in U$ and $\vec x-P_U(\vec x)\in U^\perp$ for all $\vec x\in M$. In the notation of Lemma~\ref{lem7}, $P_U(\vec a)=\vec b$ and $\vec a-P_U(\vec a)=\vec c$.

Now we can show that the poset of splitting submodules of $\mathbf M$ is isomorphic to the poset of its projections.

\begin{theorem}\label{th4}
The mappings $U\mapsto P_U$ and $P\mapsto P(M)$ are mutually inverse isomorphisms between $\mathbf L_s(\mathbf M)$ and $\BPr(\mathbf M)$.
\end{theorem}

\begin{proof}
Let $U,W\in P_s(\mathbf M)$, $P,Q\in\Pro(\mathbf M)$ and $\vec a,\vec b\in M$. Obviously, $P_U$ is a linear mapping from $\mathbf M$ to itself and $(P_U)^2=P_U$. Moreover,
\begin{align*}
(P_U\vec a)\vec b & =(P_U\vec a)((\vec b-P_U\vec b)+P_U\vec b)=(P_U\vec a)(\vec b-P_U\vec b)+(P_U\vec a)(P_U\vec b)=(P_U\vec a)(P_U\vec b)= \\
& =(\vec a-P_U\vec a)(P_U\vec b)+(P_U\vec a)(P_U\vec b)=((\vec a-P_U\vec a)+P_U\vec a)(P_U\vec b)=\vec a(P_U\vec b)
\end{align*}
showing $P_U\in\Pro(\mathbf M)$. Now $(\vec a-P\vec a)(P\vec b)=(P(\vec a-P\vec a))\vec b=(P\vec a-P\vec a)\vec b=0$ and hence $\vec a-P\vec a\in(P(M))^\perp$. This shows $P(M)+(P(M))^\perp=M$. If $P\vec a\in(P(M))^\perp$ then $(P\vec a)\vec b=(P^2\vec a)\vec b=(P\vec a)(P\vec b)=0=\vec0\vec b$ and hence $P\vec a=\vec0$ showing $P(M)\cap(P(M))^\perp=\{\vec0\}$. Hence $P(M)\in L_s(\mathbf M)$. Obviously, $P_U(M)=U$. Since $\vec a-P\vec a\in(P(M))^\perp$, we have $P=P_{P(M)}$. If $U\subseteq W$ then $P_U(M)=U\subseteq W=P_W(M)$, i.e.\ $P_U\leq P_W$. If, conversely, $P\leq Q$ then $P(M)\subseteq Q(M)$. Of course, $P_{\{\vec0\}}=\mathbf0$, $P_M=\mathbf I$ and $P_{U^\perp}=\mathbf I-P_U=(P_U)'$.
\end{proof}

The next lemma shows that the supremum of two commuting projections always exists.

\begin{lemma}\label{lem6}
Let $P,Q\in\Pro(\mathbf M)$ and assume $PQ=QP$. Then $P\vee Q=P+Q-PQ$.
\end{lemma}

\begin{proof}
We have
\[
P'Q'=(\mathbf I-P)(\mathbf I-Q)=\mathbf I-Q-P+PQ=\mathbf I-P-Q+QP=(\mathbf I-Q)(\mathbf I-P)=Q'P'
\]
and hence according to Lemma~\ref{lem5}
\[
P\vee Q=(P'\wedge Q')'=(P'Q')'=\mathbf I-(\mathbf I-P-Q+PQ)=P+Q-PQ
\]
according to Lemma~\ref{lem5}.
\end{proof}

\begin{corollary}
If $P,Q\in\Pro(\mathbf M)$ and $P\perp Q$ then $P\wedge Q=\mathbf0$ and $P\vee Q=P+Q$.
\end{corollary}

\begin{proof}
This follows from Lemmas~\ref{lem5}, \ref{lem4} and \ref{lem6}.
\end{proof}

Recall from \cite{Be} that an {\em orthomodular poset} is a bounded poset $(P,\leq,{}',0,1)$ with an antitone involution such that for all $x,y\in P$:
\[
x\vee y\text{ exists if }x\leq y,\text{ and if }x\leq y\text{ then }y=x\vee(y\wedge x').
\]
The notion of an orthomodular poset is well-defined: If $x\leq y$ then $x\vee y'$ exists and hence $x'\wedge y$ exists, too. Moreover, $x'\wedge y\leq x'$ and hence $(x'\wedge y)\vee x$ exists.

Our final result shows that the splitting submodules of $\mathbf M$ form an orthomodular poset.

\begin{theorem}\label{th5}
Let $\mathbf M$ be a module. Then $\BPr(\mathbf M)$ is an orthomodular poset, $\mathbf L_s(\mathbf M)$ is isomorphic to $\BPr(\mathbf M)$, and $U\vee W=U+W$ in $\mathbf L_s(\mathbf M)$ for every $U,W\in L_s(\mathbf M)$ with $U\perp W$ {\rm(}i.e.\ $U\subseteq W^\perp${\rm)}.
\end{theorem}

\begin{proof}
According to Lemma~\ref{lem1}, $\BPr(\mathbf M)$ is a bounded poset with an antitone involution. Now let $P,Q\in\Pro(\mathbf M)$. If $P\perp Q$ then $P\vee Q=P+Q$. If $P\leq Q$ then $P\perp Q'$, $P\vee Q'=P+Q'$, $P'\wedge Q=(P\vee Q')'$, $P\perp P'\wedge Q$ and
\[
P\vee(P'\wedge Q)=P+(P+Q')'=P+\mathbf I-(P+\mathbf I-Q)=Q.
\]
The second part of the theorem follows from Theorem~\ref{th4} and from
\[
U+W\subseteq U\vee W=(P_U+P_W)(M)\subseteq U+W
\]
for every $U,W\in L_s(\mathbf M)$ with $U\perp W$.
\end{proof}

\begin{example}
Consider the ring $(\mathbb Z_4,+,\cdot,0,1)$ of residue classes of the integers modulo $4$ and put $\mathbf M:=(\mathbb Z_4,+,\cdot,0)^2$ and $A:=\{0,2\}$. Then $\mathbf M$ has the following subspaces:
\begin{align*}
   U_1 & =\{0\}^2, \\
   U_2 & =\{0\}\times A, \\
   U_3 & =\{(0,0),(2,2)\}, \\
   U_4 & =A\times\{0\}, \\
   U_5 & =\{0\}\times\mathbb Z_4, \\
   U_6 & =\{(0,0),(0,2),(2,1),(2,3)\}, \\
   U_7 & =\{(0,0),(1,1),(2,2),(3,3)\}, \\
   U_8 & =A^2, \\
   U_9 & =\{(0,0),(1,3),(2,2),(3,1)\}, \\
U_{10} & =\{(0,0),(1,2),(2,0),(3,2)\}, \\
U_{11} & =\mathbb Z_4\times\{0\}, \\
U_{12} & =A\times\mathbb Z_4, \\
U_{13} & =\{(0,0),(0,2),(1,1),(1,3),(2,0),(2,2),(3,1),(3,3)\}, \\
U_{14} & =\mathbb Z_4\times A, \\
     M &
\end{align*}
The Hasse diagram of $(L(\mathbf M),\subseteq)$ is presented in Figure~3:

\vspace*{-2mm}

\[
\setlength{\unitlength}{7mm}
\begin{picture}(14,11)
\put(7,2){\circle*{.3}}
\put(3,4){\circle*{.3}}
\put(7,4){\circle*{.3}}
\put(11,4){\circle*{.3}}
\put(1,6){\circle*{.3}}
\put(3,6){\circle*{.3}}
\put(5,6){\circle*{.3}}
\put(7,6){\circle*{.3}}
\put(9,6){\circle*{.3}}
\put(11,6){\circle*{.3}}
\put(13,6){\circle*{.3}}
\put(3,8){\circle*{.3}}
\put(7,8){\circle*{.3}}
\put(11,8){\circle*{.3}}
\put(7,10){\circle*{.3}}
\put(7,2){\line(-2,1)4}
\put(7,2){\line(0,1)8}
\put(7,2){\line(2,1)4}
\put(3,4){\line(-1,1)2}
\put(3,4){\line(0,1)4}
\put(3,4){\line(2,1)8}
\put(7,4){\line(-1,1)2}
\put(7,4){\line(1,1)2}
\put(11,4){\line(-2,1)8}
\put(11,4){\line(0,1)4}
\put(11,4){\line(1,1)2}
\put(3,8){\line(-1,-1)2}
\put(7,8){\line(-1,-1)2}
\put(7,8){\line(1,-1)2}
\put(11,8){\line(1,-1)2}
\put(7,10){\line(-2,-1)4}
\put(7,10){\line(2,-1)4}
\put(6.675,1.25){$U_1$}
\put(2.05,3.75){$U_2$}
\put(7.4,3.75){$U_3$}
\put(11.4,3.75){$U_4$}
\put(.05,5.75){$U_5$}
\put(2.05,5.75){$U_6$}
\put(4.05,5.75){$U_7$}
\put(7.4,5.75){$U_8$}
\put(9.4,5.75){$U_9$}
\put(11.4,5.75){$U_{10}$}
\put(13.4,5.75){$U_{11}$}
\put(1.7,7.75){$U_{12}$}
\put(7.4,7.75){$U_{13}$}
\put(11.4,7.75){$U_{14}$}
\put(6.7,10.35){$M$}
\put(6.2,.3){{\rm Fig.~3}}
\end{picture}
\]
The unary operation $^\perp$ looks as follows:
\[
\begin{array}{c|ccccccccccccccccc}
   U    & U_1 &  U_2   &  U_3   &  U_4   &  U_5   &  U_6   & U_7 & U_8 & U_9 & U_{10} & U_{11} & U_{12} & U_{13} & U_{14} & M \\
\hline
U^\perp &  M  & U_{14} & U_{13} & U_{12} & U_{11} & U_{10} & U_9 & U_8 & U_7 &  U_6   &  U_5   &  U_4   &  U_3   &  U_2   & U_1
\end{array}
\]
Hence, $L_c(\mathbf M)=L(\mathbf M)$ and $L_s(\mathbf M)=\{U_1,U_5,U_6,U_{10},U_{11},M\}$. The Hasse diagram of $(L_s(\mathbf M),\subseteq)$ is depicted in Figure~4:

\vspace*{-6mm}

\[
\setlength{\unitlength}{7mm}
\begin{picture}(8,7)
\put(4,2){\circle*{.3}}
\put(1,4){\circle*{.3}}
\put(3,4){\circle*{.3}}
\put(5,4){\circle*{.3}}
\put(7,4){\circle*{.3}}
\put(4,6){\circle*{.3}}
\put(4,2){\line(-3,2)3}
\put(4,2){\line(-1,2)1}
\put(4,2){\line(1,2)1}
\put(4,2){\line(3,2)3}
\put(4,6){\line(-3,-2)3}
\put(4,6){\line(-1,-2)1}
\put(4,6){\line(1,-2)1}
\put(4,6){\line(3,-2)3}
\put(3.675,1.25){$U_1$}
\put(.05,3.85){$U_5$}
\put(2.05,3.85){$U_6$}
\put(5.4,3.85){$U_{10}$}
\put(7.4,3.85){$U_{11}$}
\put(3.7,6.35){$M$}
\put(3.2,.3){{\rm Fig.~4}}
\end{picture}
\]

\vspace*{-1mm}

One can easily see that $(L_s(\mathbf M),\subseteq,{}^\perp,\{(0,0)\},M)$ is the orthomodular lattice ${\rm MO}_2$ and hence an orthomodular poset.
\end{example}

In our examples, the poset of splitting subsemimodules or splitting submodules is a lattice. In general, this need not hold. G.~Birkhoff and J.~von~Neumann proved (\cite{BV}) that in the case of an infinite-dimensional Hilbert space over the field of complex numbers this poset is not a lattice but only an orthomodular poset.

Authors' addresses:

Ivan Chajda \\
Palack\'y University Olomouc \\
Faculty of Science \\
Department of Algebra and Geometry \\
17.\ listopadu 12 \\
771 46 Olomouc \\
Czech Republic \\
ivan.chajda@upol.cz

Helmut L\"anger \\
TU Wien \\
Faculty of Mathematics and Geoinformation \\
Institute of Discrete Mathematics and Geometry \\
Wiedner Hauptstra\ss e 8-10 \\
1040 Vienna \\
Austria, and \\
Palack\'y University Olomouc \\
Faculty of Science \\
Department of Algebra and Geometry \\
17.\ listopadu 12 \\
771 46 Olomouc \\
Czech Republic \\
helmut.laenger@tuwien.ac.at
\end{document}